\documentclass{amsart}

\usepackage{amscd}
\usepackage{amssymb}
\usepackage{stmaryrd}
\usepackage{amsthm,amsmath}

\usepackage{color}

 \usepackage{hyperref}
\hypersetup{
  colorlinks=true, 
   linkcolor=red,  
}
\usepackage[hyphenbreaks]{breakurl}

\newcommand{\bkr}{\oblong} 
\newcommand{\bigbkr}{\bigbox} 

\theoremstyle{plain}
\newtheorem{thm}{Theorem}

\newtheorem{lemma}[thm]{Lemma}
\newtheorem{theorem}[thm]{Theorem}

\newtheorem{cor}[thm]{Corollary}
\theoremstyle{definition}
\newtheorem{example}[thm]{Example}
\newtheorem{problem}[thm]{Problem}
 
\def\e{\mathbb{E \,}}
\def\BR{\mathbb{R}}
\def\BN{\mathbb{N}}

\def\p{\mathbb{P}}   
\def\pr{{\rm proj}}   
\def\ext{{\rm ext}_d} 
\def\Pr{{\rm Proj}}   
\def\Proj{\Pr}   
\def\bX{\mathbf{X}}
\def\bx{\mathbf{x}}
\def\bu{\mathbf{u}}

\def\C{{\rm Cyl}}
\def\Ext{{\rm Ext}_d} 
\def\F{{\mathcal F}}
\def\Fkn{{\mathcal{F}_n^{(k)}}}   

\begin{document}

\title{The van den Berg--Kesten--Reimer operator and inequality for infinite spaces}
 
\author{Richard Arratia \and Skip Garibaldi \and Alfred W. Hales}
\address{Arratia (corresponding author): Department of Mathematics, University of Southern California, Los Angeles, CA 90089-2532}
\email{rarratia \text{at} usc.edu}

\address{Garibaldi: Institute for Pure and Applied Mathematics, UCLA, 460 Portola Plaza, Box 957121, Los Angeles, California 90095-7121, USA}
\email{skip \text{at} member.ams.org}

\thanks{The second author's research was partially supported by NSF grant DMS-1201542.}

\address{Hales:  Center for Communications Research, San Diego, California 92121}
 

\date{\tt August 29, 2015}

\begin{abstract}
We remove the hypothesis ``$S$ is finite'' from the BKR inequality for product measures on $S^d$, which
raises some issues related to descriptive set theory.
We also discuss the extension of the BKR operator and inequality, from 2 events to 2 or more events, and we remove, in one sense, the hypothesis that $d$ be finite.

\end{abstract}

\maketitle

\tableofcontents


\section{The classic BKR inequality}
The BKR inequality,  named for 
van den Berg--Kesten--Reimer, was conjectured in \cite{BK} and proved in 
 \cite{BFiebig} and \cite{Reimer}; see \cite{CPS} or \cite{BCR} for a clear exposition.   The setup  involves a probability space of the form $S^d$, with $S$ finite, and $\p$ a product measure, and the inequality takes the form:    for two events $A,B \subset S^d$, with $A \bkr B$ for the event that, informally,  ``$A$ and $B$ occur for disjoint reasons",
$$\p(A \bkr B) \le \p(A) \p(B).
$$

The somewhat convoluted history is summarized as follows:  Kesten and van den Berg \cite{BK} defined the operation
$A \bkr B$ on subsets of $S^d$, and proved the (BK)  inequality for the special case where $A$ and $B$ are assumed to be increasing events.  Then van den Berg and Fiebig \cite{BFiebig} proved a conditional implication, not involving increasing events: ``If the inequality holds for the cases
$S^d=\{0,1\}^d$ and $\p$ is the uniform distribution, with all $2^d$ points of $S^d$ equally likely,  then the inequality holds for  any finite $S$ and any product measure on $S^d$.''  Finally, Reimer \cite{Reimer}
proved the inequality in $\{0,1\}^d$, a purely combinatorial fact, so that combined with the earlier conditional implication from
\cite{BFiebig}, the general inequality was established.

In \cite{luck}, still in the context of $S$ finite and $\p$ a product measure on $S^d$, 
we 
had a Florida-lottery-crimefighting reason to need an extension of the BKR inequality, from $r=2$ events, to the more general case $r=2,3,\ldots$.   An easy example shows that
sometimes $(A \bkr B) \bkr C \ne A \bkr (B \bkr C)$, so we gave
a natural definition for 
 the $r$-fold operator
$\bigbkr_1^r A_i$, proved that $\bigbkr_1^r A_i \subset ( \cdots ((A_1 \bkr A_2) \bkr A_3) \cdots \bkr A_r)$,
and gave the easy induction, from the classic BKR inequality, to conclude that
$$
    \p\left (\bigbkr_1^r A_i \right) \le \prod_1^r \p(A_i).
$$
Although the case $S$ finite was sufficient for our application, it seemed strange to have to quote the hypothesis ``$S$ is finite'', before invoking the inequality.  Indeed, the  first draft of 
\cite{luck} made the mistake of omitting this hypothesis --- but thankfully was called to the carpet by a referee.

In this paper, we remove the restriction that $S$ be finite,  allowing $S=\mathbb{N}$ or $S=\mathbb{R}$, along with an arbitrary product probability measure on $S^d$, for our main result,
Theorem \ref{thm R}.   This raises issues related to descriptive set theory;  the BKR combination of Borel sets need not be a Borel set, and the BKR combination of Lebesgue measurable sets need not be Lebesgue measurable, see Example \ref{example hales}.  We will also, in Section \ref{sect infinite product}, remove the restriction that $d$ be finite, for one of the two natural ways of generalizing the BKR operator to spaces of the form $S^\BN$.  

Other extensions and complements to the BKR inequality are given in
\cite{alexander,GoldsteinRinott,KahnSaksSmyth}.
In greater detail, \cite{GoldsteinRinott} gives a generalization of the BKR operator and inequality
which applies to spaces such as $\BR^d$;  however, the combination of sets
which \cite[formula (5)]{GoldsteinRinott}  identifies as ``$A$ and $B$ occur for disjoint reasons"
is somewhat different from the original BKR combination $A \bkr B$, and depends on
the choice of measure and notions of essential infimum.   It is easy to
see the the BKR combination of events from \cite{GoldsteinRinott} is a superset of
the standard $A \bkr B$, hence the result from \cite{GoldsteinRinott},  
with the corrections and improvements provided in  \cite{GRarxiv},
proves the outer measure assertion in our Theorems \ref{thm [0,1]} and \ref{thm R}, via a method which finesses all issues of
projective sets by an appeal to Tonelli's theorem.     
In contrast to their approach, ours extends the BKR operator and inequality to infinite spaces in a way that closely follows the original definitions, meaning as a combination of events, rather than a combination of events and measures.
 
\subsection*{Acknowlegement}

We thank Yiannis Moschovakis for a helpful conversation.
 
 \section{Definition of the BKR operators}

The formal definition of $A \bkr B$, copied from \cite{BK}, begins with the notation
 $\omega= (\omega_1,\ldots,\omega_d)$ or $\overline\omega=(\overline\omega_1,\ldots,\overline\omega_d)$ for elements of $S^d$.  For  $\omega\in S^d$ and $K\subset [d]:= \{1,\ldots,d\}$,  consider the \emph{thin cylinder} $\C(K,\omega) :=\{\overline\omega \! : \ \overline\omega_i=\omega_i,\ i\in K\}$. For $A,B\subset S^d$  define $A \bkr B$ as the set of $\omega$ for which there exists a $K\subset [d]$ such that $\C(K,\omega)\subset A$ and $\C(K^c,\omega)\subset B$, where
$ K^c := [d] \setminus K$ is the complement of $K$ relative to the universe of coordinate indices.

A small paraphrase of this definition is based on $[A]_K$ defined to be the largest cylinder 
set\footnote{Both  $\C(K,\omega)$ and $[A]_K$ are defined relative to  $S^d$.   We have several occasions  in this paper to work simultaneously with two different sets in the role of $S$, and it should be understood that the definition of the BKR operator for sets $A,B \subset S^d$ also involves the choice of $S$ and $d$.   Apart from Section \ref{sect relax}, we use the same symbol
$\bkr$ for every operator of this form, and leave it to the reader to understand the appropriate context.} 
 contained in $A$ and free in the directions indexed by $K^c$:
\begin{equation}\label{def maximal cylinder}
  \text{for } A \subset S^d, \ \   [A]_K := \{ \omega \! : \ \C(K,\omega)\subset A \}.
\end{equation}
With this notation,

\begin{equation}\label{def bkr kesten}
  \text{for } A,B \subset S^d, \ \     A \bkr B := \bigcup_{ K \subset [d]} [A]_K \cap [B]_{K^c}.
\end{equation}
An obvious relation, that $ J \subset K \subset [d]$ implies $[A]_J \subset [A]_K$, shows that
\eqref{def bkr kesten} is equivalent to the following:
\begin{equation}\label{def bkr 2}
  \text{for } A,B \subset S^d, \ \        A \bkr B := \bigcup_{\text{ disjoint }  J,K \subset [d]} [A]_J \cap [B]_K.
\end{equation}

The  definition of the simultaneous $r$-fold BKR operator given in \cite{luck} is, for $A_1,\ldots,A_r \subset S^d$, 
\begin{equation}\label{def bkr many}
 \bigbkr_{1 \le i \le r} A_i  \equiv    A_1 \bkr A_2 \bkr \cdots \bkr A_r  
 := \bigcup_{J_1,\ldots, J_r} 
\  [A_1]_{J_1} \cap [A_2]_{J_2} \cap \cdots \cap [A_r]_{J_r},
\end{equation}
where the union is taken over 
disjoint  subsets $J_1,\ldots, J_r$ of $\{1,\ldots,d\}$.
It  is clear that for the case $r=2$, definition
\eqref{def bkr many} agrees with  \eqref{def bkr 2}, and hence with \eqref{def bkr kesten}.

\subsection{Careful notation for cylinders, projections, extensions}
\label{sect proj notation}

We follow the strict convention that, for any sets $U,V$, the set $U^V$ is the set of all functions from $V$ to $U$,
and an element $f \in U^V$ carries the information:  what is the domain of $f$, and what is the range of $f$.   For the case $V=\emptyset$, there is one point exactly in $U^V$.
Since we use the notational convention, common in combinatorics, that for $d=0,1,2,\ldots$,  $[d]:=\{1,2,\ldots,d\}$, the $d$-fold Cartesian product of a set $S$ with itself, $S^d$, is exactly equal to $S^{[d]}$.   But for  $0 \le k \le d$, there are ${d \choose k}$ subsets $K \subset [d]$, with $|K|=k$, and there are ${d \choose k}$ different sets $S^K$; only one of these is equal to $S^k$, namely, the one with $K =[k]$. 

It will be convenient to work first with the case $S=[0,1]$, allowing us to specialize to the uniform distribution.

For $ K \subset [d]$,  the projection
$$
    \pr_K: [0,1]^{[d]} \to [0,1]^K
$$
is, naturally, the function $f \mapsto f|_K$ which restricts a function $f \in     [0,1]^{[d]}$ to have domain $K$.   There is a single one-to-many \emph{relation} $\ext$, with domain $\cup_{K \subset [d]} [0,1]^K$, which serves as the inverse for \emph{all} of maps $\pr_K$,  namely, $(g,f) \in \ext$ if and only if, for some  $K \subset [d]$,
$g \in [0,1]^K$, $f \in [0,1]^{[d]}$, and $g=f|_K$. 
 
For any set $D$, we write $2^D$ for the power set of $D$, i.e., the set of all subsets of $D$. 
We will be fussy, to distinguish a function from $D$ to $D'$, and its inverse relation,  written with lowercase,
from the induced functions, mapping 
$2^{D}$ to $2^{D'}$ and back, written with uppercase.

Thus, we have $2^d $ projection functions
$$ 
   \Pr_K :   2^{[0,1]^{[d]}}   \to 2^{[0,1]^K},
$$
and a single extension  function,
$$
    \Ext:   2^{\cup_K [0,1]^K} \to 2^{[0,1]^{[d]}}.
$$
In particular, for  $K \subset [d]$, 
$$
\text{ for } C \subset [0,1]^K,   \Ext(C) := \{ f \in [0,1]^{[d]} \! : \    f|_K \in C \}   = \Proj_K^{-1}(C)   .
$$
With $k=|K|$, if $C$ is Borel 
then so is $\Ext(C)$, and $m_k(C)=m_d(\Ext(C))$,
and if $C$ is Lebesgue measurable then so is $\Ext(C)$, and 
$\lambda_k(C)=\lambda_d(\Ext(C))$  --- see Section \ref{sect m notation} for our notation for Lebesgue measures.

\section{Measurability considerations}

\subsection{Introductory motivation}
In 1905, Lebesgue  \cite[pages 191--192]{1905} stated, incorrectly, that projections of Borel sets are Borel sets, and Suslin \cite{1917} showed otherwise.   Superficially, this is an obstacle to extending
the BKR inequality from $S^d$ with $S$ countable to the case with $S=[0,1]$, since in $[0,1]^d$, even starting with Borel sets $A,B$, we cannot assert that $A \bkr B$ is also a Borel set.  In more detail,
 $$
 A \bkr B :=\cup_{ K \subset [d]} [A]_K \cap [B]_{K^c}
 $$
where $[A]_K$ is the maximal cylinder subset of $A$ free in the directions in $[d] \setminus K$,
equivalently, using notation from Section \ref{sect proj notation},
\begin{equation}\label{starting to be careful}
    [A]_K := \Ext \left( [0,1]^{K} \setminus \Pr_K(A^c) \right).
\end{equation}    
However, 
Suslin 
also
showed that projections of Borel sets 
are nice,
 in the concrete sense of having equal inner and outer measure, i.e., being measurable in the \emph{completion} of the Borel sigma-algebra with respect to Lebesgue measure \cite[8.4.1]{cohn}.  
For history, see  \cite[p 500]{Dudley},\cite[p 232]{Potter},\cite{YNM,AK}.

\subsection{Notation:   $m_d$ versus $\lambda_d$}\label{sect m notation}
It is a very common and confusing practice to use the name \emph{Lebesgue measure},  here in the context of $[0,1]^d$, to refer to \emph{two different} objects.  The first object called Lebesgue measure is the measure on the Borel sets of $[0,1]^d$, determined by the requirement that it extends the notion of volume for solid rectangles
$[a_1,b_1] \times \cdots \times [a_d,b_d]$. We shall use the notation $m_d$ for this first measure, so that for a Borel set $A \subset [0,1]^d$, we may write $m_d(A)$.   The second object called Lebesgue measure is the completion of the first object; we shall use the notation $\lambda_d$ for this measure.   Hence, the sentence
$$
     x=\lambda_d(B)
$$
is shorthand for the statement that (there exist Borel sets $A,C \subset [0,1]^d$
with $A \subset B \subset C$ and $m_d(A)=m_d(C)=x$).  There is no ambiguity in the phrase \emph{Lebesgue measurable}, since this describes elements of the completed sigma-algebra, which is the domain of $\lambda_d$.

The  ${d \choose k}$ different spaces $[0,1]^K$, for $K \subset [d]$ with $|K|=k$, are all naturally measure isomorphic to $[0,1]^k$. Rather than writing the explicit isomorphism, or naming  the corresponding copies of Lebesgue measures as $m_K$ and $\lambda_K$, we simply write $m_k$ and $\lambda_k$.  This is a minor abuse of notation, and not a capital crime.

\subsection{Details for measurability}

\begin{lemma}\label{lem cylinder}
Assume that $A$ is a Borel subset of $[0,1]^d$, and that $  K \subset [d]$.   Then, the cylinder $[A]_K$ is Lebesgue measurable.   
 If $A,B$ are Borel subsets of 
$[0,1]^d$, then $A \bkr B$ is Lebesgue measurable, and 
 if  $A_1,\ldots,A_r$ are Borel subsets of 
$[0,1]^d$, then $\bigbkr_1^r A_i$ is Lebesgue measurable.
\end{lemma}     
\begin{proof}
To start, $A$ is a Borel subset of
$[0,1]^d$ so $A^c := [0,1]^d \setminus A$ is also Borel, and the projection
$ C := \Pr_K (   A^c)$ is an analytic subset of $[0,1]^K$.   Since analytic sets
are Lebesgue measurable, there exist Borel subsets $B,D \subset  [0,1]^K$
with
$$
     B \subset C \subset D, \ \  m_k(B)=m_k(D) =: 1-x.
$$     
Taking complements relative to $[0,1]^K$, we have 
$$      
 D^c  \subset C^c \subset B^c, \ \  m_k(D^c)=m_k(B^c) =  x. 
$$       
Let $E = \Ext(D^c), F = \Ext(B^c)$, so that $E$ and $F$ are Borel subsets of
$[0,1]^d$, with 
$$
    E \subset [A]_K \subset F, \ \  m_d(E) = m_d(F)=x.
$$
This shows that  $[A]_K$ is Lebesgue measurable, with $\lambda_d( [A]_K)=x$.

The Lebesgue measurability claims for $A \bkr B$  and $\bigbkr_1^r A_i$
now follow immediately from the definitions \eqref{def bkr kesten} and \eqref{def bkr many}.
\end{proof}

The following example shows why, in Corollary  \ref{cor [0,1]}, with the hypothesis that $A$ and $B$ are Lebesgue measurable, we could not simply state that $\lambda_d(A \bkr B) \le \lambda_d(A) \, \lambda_d(B)$.

\begin{example}\label{example hales}
  The BKR combination of Lebesgue measurable sets need not be Lebesgue measurable, 
  as shown by this example with $d=2$.  Take a set $C \subset [0,1]$ which is \emph{not}  a Lebesgue measurable
subset of [0,1].  Then $C^2 \subset [0,1]^2$ is  \emph{not}  a Lebesgue measurable
subset of  $[0,1]^2$. The diagonal in $[0,1]^2$ is 
$$
    D  := \{ (x,x) \! : \ x \in [0,1]  \} \  \subset [0,1]^2,
$$
and this is a Borel subset of $[0,1]^2$, with $m_2(D)=0$.  Hence the set
$$  
    E :=  \{ (x,x) \!: \  x \in ([0,1]\setminus C)  \}  \  \subset D \subset [0,1]^2
$$
\emph{is} Lebesgue measurable, with $\lambda_2(E)=0$.
Now, taking complement relative to $[0,1]^2$, let
$$
    A  := [0,1]^2 \setminus E,
$$    
so that $A$ is Lebesgue measurable, with $\lambda_2(A)=1$.
We have
$$
    [A]_{\{1\}} = C \times [0,1]   , \ \ \  [A]_{\{2\}} = [0,1] \times C,       
$$ 
and with $B :=A$ we have
$$ 
    A \bkr B = C^2.
$$
\end{example}

\section{Approximation,  from [0,1] to a finite set}

\begin{theorem}\label{thm [0,1]}
For Borel subsets $A,B$ in $[0,1]^d$,
$$
 \lambda_d(A \bkr B) \le m_d(A) \, m_d(B).
$$
\end{theorem}

\subsection{Overview of the argument}\label{sect approximate}.
We want to prove that, for Borel 
$A,B \subset [0,1]^d$, we have $\lambda_d( A \bkr B) \le m_d(A) \ m_d(B)$,
and we proceed by contradiction.   Thus, we assume that we have $A,B$ with
\begin{equation}\label{def epsilon}
    4 \, \varepsilon := \lambda_d( A \bkr B) - m_d(A) \ m_d(B)   > 0,
\end{equation}
and we work to provide an example, with finite $S$, in which the classic BKR inequality on $S^d$ is violated.

In this example,  for some large but finite $n$, we have $|S| = 2^{n}$, 
$|S^d|=2^{nd}$,  corresponding to the number of atoms in the
``observe the first $n$ bits'' sigma-algebra $\F_n^{(d)}$ on $[0,1]^d$.   The product measure $\p$ on $S^d$ will be the uniform distribution, with mass $2^{-nd}$ at each point of $S^d$.  We will produce subsets
$A'', B'' \subset S^d$ for which
\begin{equation}\label{goal 1}
   \p(A'' \bkr B'') \ge    \lambda_d( A \bkr B) - \varepsilon
\end{equation}
and
\begin{equation}\label{goal 2}   
    \p(A'') \le m_d(A) + \varepsilon, \ \p(B'') \le m_d(B)+ \varepsilon,
\end{equation}
so that $A'',B''$ violate the classic BKR inequality.

\subsection{Set approximation, in 1 dimension}
To lighten the notational burden, we start with  dimension 1, and  review a familiar martingale, from for example \cite[Examples 35.3, 35.10]{Billingsley3}. 
The probability space is [0,1], with the Borel sigma-algebra, and the probability measure is $m_1$.
For $n=0,1,2,\ldots$, define  $\F_n$ to be the sigma-algebra generated by the $2^n$ disjoint intervals, $[0,1/2^n)$, $[1/2^n,2/2^n),\ldots$,
$[(n-2)/2^n,(n-1)/2^n)$, $[1-1/2^n, 1]$.  Note that the last of these intervals is exceptional, in that it is closed at both ends,
but all $2^n$ intervals $I$ have length $m_1(I)=1/2^n$.      The sigma-algebra $\F_n$ has $2^n$ atoms, and is a family of $2^{2^n}$ subsets of [0,1].  These sigma-algebras are nested,
and $\sigma( \cup_{n \ge 0}  \F_n)$  is the usual Borel sigma-algebra on [0,1]. 

Hence for any Borel measurable $h:[0,1] \to [0,1]$,   $M_n :=   \e ( h | \F_n)$ is a martingale. 
  Explicitly, on an atom $I$ of $\F_n$,
$M_n =2^n \e(h;I) = 2^n \int_I h(x)  \ dx$. The martingale convergence theorem implies that $M_n$ converges to $h$, almost surely and in $L_1$, with the $L_1$ convergence meaning that   $\e | M_n - h | \to 0$ as $n \to \infty$.

In particular, given a Borel measurable $C \subset [0,1]$,  we take $h$ to be the indicator function $h = 1_C$.
Explicitly,   on an atom $I$ of $\F_n$,   $M_n = 2^n \ m_1( C \cap I)$.
From this martingale, we round values in [0,1/2]  down to 0, and values in (1/2,1]  up to 1, to get a deterministic set $C_n \in \F_n$.   Explicitly, 
$$
C_n :=  \{ \omega \in [0,1]: M_n(\omega) > 1/2 \}.
$$
For a point $x$ to be in the symmetric difference set, $C \Delta C_n$,  the rounding error is at least one half.  This implies that   $m_1( C \Delta C_n) \le 2 \e | M_n - 1_C |$.

\subsection{Set approximation, in $k$ dimensions}

The above extends to dimension $k$, for $k=1,2,\ldots$, with no difficulties, only extra notation.  The probability space is $[0,1]^k$,  with 
$m_k$ serving as the probability measure.  We define, for $n=0,1,2,\ldots$, the analogous  sigma-algebra $\Fkn$  with $2^{nk}$ atoms, and for any Borel measurable set 
$C\subset [0,1]^k$, the martingale argument gives us determinstic sets $C_n$, with
\begin{equation}\label{approximate k}
    m_k ( C \Delta C_n )  \to 0,
\end{equation}
and $C_n$ is $\Fkn$ measurable.

\subsection{Approximation in $[0,1]^d$ to control BKR ingredients}

Recall that for $A \subset [0,1]^{d}$ and   $K \subset [d]$,  $[A]_K \subset  [0,1]^{d}$ is the  (maximal) cylinder subset of $A$, in the directions not restricted by $K$.

We write 
\begin{equation}\label{def cylinder base}
  [[A]]_K :=    \Pr_K ( [A]_K) \ =
[0,1]^K \setminus  \Pr_K( [A^c]_K) \ \  \subset [0,1]^K
\end{equation}
for the base of this cylinder.  From the proof of Lemma \ref{lem cylinder}, $[[A]]_K$ is Lebesgue measurable, and there is a Borel subset
$C \subset [0,1]^K$ with 
\begin{equation}\label{borel inside}
    C \subset [[A]]_K,  \ m_k(C)=\lambda_k( [[A]]_K).
\end{equation}
Observe that, with $1 \le k = |K| < d$,
$$
     [A]_K = \Ext( [[A]]_K) \supset \Ext(C)
$$ 
and
$$
       \lambda_d( [A]_K) = \lambda_k([[A]]_K) =  m_k(C) = m_d(\Ext(C)).
$$
          
Taking 
$A$ or $B$, and  $K \subset [d]$, we have $2^{d+1}$ instances of
a set $C \subset [[A]]_K$ or      $C \subset [[B]]_K$, with $0 \le k := |K| \le d$, to serve as the target for an approximation
as given by the martingale argument, summarized by \eqref{approximate k}.   Since 
\begin{equation}\label{A bkr B}
    A \bkr B = \bigcup_K \, [A]_K \cap [B]_{K^c},
\end{equation}    
has  $2^{d+1}$ ingredients, we take
$$
\delta := \varepsilon / 2^{d+1},
$$
and pick a single value of $n$ so that for each of the instances of $C$,
\begin{equation}\label{delta close}
    m_k(C \Delta C_n) < \delta.
\end{equation}    
    
When $C \subset [[A]]_K$, the dyadic approximation $C_n$ is a subset of $[0,1]^K$, and we write
$$
    A_{n,K} := \Ext( C_n) \subset [0,1]^d
$$ 
for the cylinder set whose base is $C_n$.
Thus, with similar notation for $B$ and approximations $B_{n,K}$ to $[B]_K$, we have, from \eqref{borel inside} and \eqref{delta close},
that
\begin{equation}\label{not too small}
   \lambda_d( [A]_K \setminus A_{n,K} ) < \delta,\ \  \lambda_d(  [B]_K\setminus B_{n,K} ) < \delta,
\end{equation}
and since $[A]_K \subset A, [B]_K \subset B$,
\begin{equation}\label{not too big}
   m_d( A_{n,K} \setminus A ) < \delta,\ \  m_d( B_{n,K} \setminus B ) < \delta.
\end{equation}

Note also that $A_{n,K},B_{n,K} \in \F_n^{(d)}$.   We take
\begin{equation}\label{def A'}
    A' := \cup_K A_{n,K}, \ \ B' := \cup_K B_{n,K},
\end{equation}
so that     
$$  
   A',B' \in \F_n^{(d)},
$$   
and for every $K$, $[A']_K \supset A_{n,K}$, similarly for $B$, so that by \eqref{not too small},
\begin{equation}\label{not too small again}
\lambda_d( [A]_K \setminus [A']_K ) < \delta,\ \  \lambda_d(  [B]_K\setminus [B']_K ) < \delta.
\end{equation}

Using \eqref{not too small again},
$$
\lambda_d( [A']_K \cap [B']_{K^c})  >   \lambda_d(    [A]_K \cap [B]_{K^c}) - 2 \delta,
$$
and hence for the unions, with $2^d$ values for $K$, using $2^{d+1} \delta = \varepsilon$,
$$ 
   \lambda_d( A' \bkr B') > \lambda_d(A \bkr B) - \varepsilon.
$$
To get an inequality in the opposite direction, 
combining \eqref{not too big} with \eqref{def A'}, 
$$    
      m_d( A' \setminus A) < 2^d \, \delta < \varepsilon, \text{ hence } m_d(A') < m_d(A) + \varepsilon,
$$
and similarly         $m_d(B') < m_d(B) + \varepsilon$.

Finally, since $A',B' \in    \F_n^{(d)}$,   we take equivalence classes modulo the atoms of $\F_n^{(d)}$,
to produce our sets $A'',B'' \in S^d$ for $S$ with $|S|=2^n$, to get the example satisfying
\eqref{goal 1} and \eqref{goal 2}.   This completes a proof of Theorem \ref{thm [0,1]}.

\begin{cor}\label{cor [0,1]}
For Lebesgue measurable $A,B \subset [0,1]^d$,  there exists a Borel set $C$, with
$$
 (A \bkr B) \subset C, \ \ \ m_d(C)    \le \lambda_d(A) \, \lambda_d(B).
$$
\end{cor}
\begin{proof}
Take Borel sets $A_1,B_1 \subset [0,1]^d$ with $A \subset A_1, B \subset B_1$, and $m_d(A_1)=\lambda_d(A), m_d(B_1)=\lambda_d(B)$.   Obviously $A \bkr B \subset A_1 \bkr B_1$, and Theorem \ref{thm [0,1]} implies the existence of a Borel set $C$ with $A_1 \bkr B_1 \subset C$ and
$m_d(C) \le m_d(A_1) m_d(B_1)$.
\end{proof}

\section{Extension to 3 or more events}

\begin{theorem}\label{thm [0,1] many}
For Borel subsets $A_1,\ldots,A_r$ in $[0,1]^d$,
\begin{equation}\label{goal multi}
 \lambda_d \left( \bigbkr_1^r A_i \right) \le \prod_1^r m_d(A_i).
\end{equation}
For Lebesgue measurable $A_1, ..., A_r$ in $[0,1]^d$, there exists a Borel set $D$ with $\bigbkr_1^r A_i \subset D$ and $m_d(D) \le \prod \lambda_d(A_i)$.

\end{theorem}
\begin{proof}
Define  sets $B_1,B_2,\ldots,B_r\subset [0,1]^d $, Lebesgue measurable sets $C_1,C_2,\ldots,C_r \subset [0,1]^d$, and Borel sets $D_1,D_2,\ldots,D_r$  
recursively, with
$$
   A_1 = B_1 = C_1 = D_1
$$
and for $i=2$ to $r$, using Lemma \ref{lem cylinder},
$$
     B_i = B_{i-1} \bkr A_i,
$$
$$ 
C_i = D_{i-1} \bkr A_i,
$$ 
$$ 
D_i \text{ is a Borel set with }
         C_i \subset D_i, \ \lambda_d(C_i) = m_d(D_i). 
$$

The BKR monotonicity relation that $B \subset D$ implies $B \bkr A \subset D \bkr A$, and induction, shows that for all $i$, $B_i \subset C_i \subset D_i$.
We check that $C_i$ is Lebesgue measurable by noting the it is the BKR combination of
two Borel sets, namely $D_{i-1}$ and $A_i$.

Theorem \ref{thm [0,1]} implies that $\lambda_d(C_i) \le m_d(D_{i-1}) \, m_d(A_i)$, and together with the defining property of $D_i$ this yields
$$
\lambda_d(C_i) \le \lambda_d(C_{i-1}) \, m_d(A_i)
$$
and it follows by induction that $\lambda_d(C_r) \le   \prod_1^r m_d(A_i)$.

It is shown in \cite{luck} that $ \bigbkr_1^r A_i  \subset B_r$.

Combined with $B_r \subset C_r$, we have $ \bigbkr_1^r A_i  \subset C_r$.  Lemma \ref{lem cylinder} shows that
$ \bigbkr_1^r A_i$ is Lebesgue measurable, so we have proved \eqref{goal multi}.

The case with Lebesgue measurable inputs $A_1,\ldots, A_r$ now follows from the Borel case, by the same reasing used to derive Corollary \ref{cor [0,1]} from Theorem \ref{thm [0,1]}.
 \end{proof}

\section{Extension of the BKR inequalities to $\mathbb{R}^d$}

Say we are given a product probability measure $\p$ on $\mathbb{R}^d$.   This is equivalent to saying that $\p$ is the law, with the Borel sigma-algebra on $\BR^d$, of $\bX = (X_1,\ldots,X_d)$, with $X_1,X_2,\ldots,X_d$ mutually independent, \emph{and} with \emph{some} given marginal distributions --- given by, say, the cumulative distribution functions $F_i$, where $F_i(t) := \p(X_i \le t)$ for $-\infty < t < \infty$.
Let $G_i$ be what is commonly called ``$F_i^{-1}$, the inverse cumulative distribution function for $X_i$'', 
or 
``the quantile function for the distribution of $X_i$''.   Specifically, we take the domain of $G_i$ to be (0,1),
and for $0 < u < 1$,
$$
   G_i(u) := \sup \{x \! : \ \p(X_i \le  x) \le u \},
 $$  
 this being a choice that makes $G_i(\cdot)$ right-continuous.   
It is standard to use this in a coupling:  with $U$ uniformly distributed in (0,1), $G_i(U)$ is equal in distribution to $X_i$. 

The net effect of this is to reassure the reader we have  no claim to originality, if we \emph{define}
\begin{equation}\label{def g}
   g: (0,1)^d \to \mathbb{R}^d, \ \ \bu = (u_1,\ldots,u_d) \mapsto \bx :=(G_1(u_1),\ldots,G_d(u_d)).
\end{equation}
Also, it is \emph{obvious} that under the uniform distribution on $(0,1)^d$, $g(\omega)$ is equal in distribution to $\bX$, i.e., for every Borel set $A$ in $\BR^d$, $m_d(g^{-1}(A))=\p(A)$.

\begin{theorem}\label{thm R}
For Borel subsets $A,B$ of  $\BR^d$, under any complete product probability measure $\p$ on
$\BR^d$,
\begin{equation}\label{real conclusion}
 \p(A \bkr B) \le \p(A) \, \p(B).
\end{equation}
For Borel subsets $A_1,\ldots,A_r$ of $\BR^d$, under any complete  product probability measure $\p$ on
$\BR^d$,
\begin{equation}\label{real conclusion many}
 \p(\bigbkr_1^r A_i ) \le \prod_1^r \p(A_i).
\end{equation}
\end{theorem}
\begin{proof}
The map $g$ defined by \eqref{def g} is Borel measurable.  Since the $i$th coordinate of $g(\bu)$
depends only on $u_i$, the BKR operators respect $g$, that is,
\begin{equation}\label{being careful}
    \text{for } a := g^{-1}(A) , b := g^{-1}(B) \subset [0,1]^d, \ \ \ a \bkr b = g^{-1}( A \bkr B).
\end{equation}
Of course, the BKR operator  $\bkr$ appearing in $a \bkr b$ in \eqref{being careful} is defined for $[0,1]^d$ by \eqref{def bkr kesten} and \eqref{starting to be careful}, while the BKR operator $\bkr$ appearing in $A \bkr B$ in \eqref{being careful} is defined for $\BR^d$ by the appropriate analog of 
\eqref{starting to be careful};  these \emph{are} different operators. 

Now apply Theorem \ref{thm [0,1]} to get \eqref{real conclusion}.  For the $r$-fold BKR operator, the same $g$, combined with Theorem \ref{thm [0,1] many}, implies \eqref{real conclusion many}.
\end{proof} 

\begin{cor}\label{cor subset R}
Suppose $S \subset \BR$ is Borel measurable.   For Borel subsets $A,B$  and $A_1,\ldots,A_r$ of $S^d$, under any complete product probability measure $\p$ on $S^d$, \eqref{real conclusion} and \eqref{real conclusion many} hold.
\end{cor}
\begin{proof}
Extend $A \subset S^d$ to $\hat{A} \subset \BR^d$ given by $\hat{A} := A \cup (\BR^d \setminus S^d)$, likewise extend $B$ or $A_1,...,A_d$, and apply Theorem \ref{thm R}.
\end{proof}

\section{Infinite products}\label{sect infinite product}
How should the BKR operator be extended from $S^d$ to $S^\infty \equiv S^\BN$?
For $A \subset S^\BN$, and $K \subset \BN$, the definition of $[A]_K$ extends in the obvious way from \eqref{def maximal cylinder}:  
$[A]_K$
is the maximal cylinder subset of $A$, free in all coordinates indexed by $\BN \setminus K$. 

Definition \eqref{def bkr 2} for the BKR operator $\bkr$ on spaces of the form $S^d$, if modified to apply to $S^\BN$ merely by replacing $[d]$ by $\BN$, yields an operator we shall call $
\bkr_{=\infty}$:
\begin{equation}\label{def bkr = infty}
    A \bkr_{=\infty} B := \bigcup_{\text{ disjoint }  J,K \subset \BN} [A]_J \cap [B]_K.
\end{equation}
One problem with this operator is that it involves an uncountable union, so in the measurability
argument from Lemma \ref{lem cylinder}, 
the cylinders such as $[A]_J$ are Lebesgue measurable, but this fails to imply that for Borel set $A,B$, the result $A \bkr_{=\infty} B$ is Lebesgue measurable.  A more severe problem with
definition \eqref{def bkr = infty} is that it does not seem to yield to any approximation scheme down to a known version of the BKR inequality, as in the heart of this paper, Section 
\ref{sect approximate}.

Hence, for spaces of the form $S^\BN$, we adopt the following definitions:
\begin{equation}\label{def bkr infty}
   \text{for } A,B \subset S^\BN, \ \  A \bkr B := \bigcup_{\text{finite disjoint }  J,K \subset \BN} [A]_J \cap [B]_K
\end{equation}
and for $A_1,\ldots,A_r \subset S^\BN$, 
\begin{equation}\label{def bkr r infty} 
 \bigbkr_{1 \le i \le r} A_i  \equiv  A_1 \bkr \cdots \bkr A_r := 
\bigcup_{ \text{finite disjoint } J_1,\ldots,J_r \subset \BN} \  \bigcap_1^r [A_i]_{J_i}.
\end{equation}
It \emph{may} have been nice to use the customary BKR symbol $\bkr$ in the above definitions, rather than contrive new notation, perhaps $\bkr_{finite}$ or $\bkr_\infty$.  It \emph{is} valid,
and 
would
allow a single universal definition, to replace all of \eqref{def bkr 2}, \eqref{def bkr many},
\eqref{def bkr infty}, and \eqref{def bkr r infty}:  for countable index set $I$ (such as $I=[d]$ or $I=\BN$), for $r \ge 2$ and 
for $A_1,\ldots,A_r \subset S^I$, we define the event \emph{that $A_1,\ldots,A_r$ occur for \emph{finite} disjoint
sets of reasons},  
\begin{equation}
\bigbkr_1^r A_i := A_1 \bkr \cdots \bkr A_r := 
\bigcup_{ \text{finite disjoint } J_1,\ldots,J_r \subset I} \  \bigcap_1^r [A_i]_{J_i}.
\end{equation}
However, in light of the natural alternate extension given by \eqref{def bkr = infty}, users of the symbol $\bkr$ in the context of infinite products spaces \emph{should} attach warning prose, as we do in Theorems \ref{thm [0,1] N} and \ref{thm R N} below.
 
\begin{example}
Consider $(\Omega,\mathcal{F},\p)$ with $\Omega=[0,1]^\BN$,
$\mathcal{F}=$ the Borel sets, and $\p=m$, Lebesgue measure;  as usual let $X_i:=$ the $i$th coordinate, $S_n := X_1+\cdots+X_n$.  Let $A = \{ \limsup S_n/n \ge .2\}$.  Then
$A \bkr B=\emptyset$ for \emph{every} event $B$, but  \mbox{$\p(A \bkr_{=\infty} A)=1$}, which can be seen by taking $J=$ the odd positive integers, $K=$ the even positive integers.  Consider the $r$-fold $\bkr_{=\infty}$ operator defined in the natural way.  Take $A_1=\cdots=A_r =A, \ $ $B_r :=
A_1 \bkr_{=\infty} \, A_2 \bkr_{=\infty} \, \cdots \bkr_{=\infty} \, A_r$, and $C_r :=
A_1 \bkr A_2 \bkr \cdots \bkr A_r \equiv \bigbkr_1^r A_i$ . We have $B_r = \emptyset$ if and only if $r>5$, and  $\p(B_1)=\p(B_2)=1,\p(B_3)=\p(B_4)=\p(B_5)=0$.    We have $C_r =\emptyset$ for $r=1,2,3,\ldots$  with the case $r=1$ serving to  highlight a difference between the two forms of notation,
$\bigbkr_1^r A_i$ and $ A_1 \bkr \cdots \bkr A_r$ --- does the latter reduce to $A_1$ when $r=1$? Of course not.
\end{example}

The extension of theorems \ref{thm [0,1]}  and  \ref{thm [0,1] many} from $[0,1]^d$ to $[0,1]^\BN$ is relatively easy.
Following the notational scheme from Section \ref{sect m notation}, we write $m$ for Lebesgue measure on the Borel subsets of $[0,1]^\BN$, and $\lambda$ for the completion of $m$, so $\lambda$ is Lebesgue measure on the Lebesgue measurable subsets of $[0,1]^\BN$.

\begin{theorem}\label{thm [0,1] N}
Consider the BKR combination of events, that they occur for \emph{finite} disjoint sets of reasons,
as specified by \eqref{def bkr infty} and \eqref{def bkr r infty}.
 For Borel subsets $A,B$ in $[0,1]^\BN$,  $A \bkr B$ is Lebesgue measurable, and
$$
 \lambda(A \bkr B) \le m(A) \, m(B).
$$
For Borel subsets $A_1,\ldots,A_r$ in $[0,1]^\BN$, $A_1 \bkr \cdots \bkr A_r$ is Lebesgue measurable, and
\begin{equation}\label{goal multi N}
 \lambda \left( \bigbkr_1^r A_i \right) \le \prod_1^r m(A_i).
\end{equation}

\end{theorem}
\begin{proof}
The Lebesgue measurability of the BKR products is clear from the
sentence following \eqref{def bkr  = infty}. Define the level-$d$ BKR operator on $[0,1]^\BN$ by 
\begin{equation}\label{def bkr d}
    A \bkr_d B := \bigcup_{\text{ disjoint }  J,K \subset [d]} [A]_J \cap [B]_K.
\end{equation}
It is obvious that $A \bkr B$ is the countable, nested union of these, hence
$$   A \bkr B = \cup_{d \ge 0} A \bkr_d B \ \ \text{ and } 
\lim_{d \to \infty} \lambda(A \bkr_d B) \  = \lambda(A \bkr B).
$$
Therefore, 
 it suffices to show that for $d < \infty$, $\lambda(A \bkr_d B) \le m(A) \, m(B)$.

Fix $d$ and let $C = A \bkr_d B$.   Extend the notation $[[A]]_K$
for the base of the cylinder $[A]_K$,  from \eqref{def cylinder base} to the situation with $A \subset [0,1]^\BN$, and apply it with $K=[d]$.  
Take  $A' := [[A]]_{[d]} \subset [0,1]^d$,  so  $[A]_{[d]} \subset A$, and $\lambda_d(A') = \lambda( [A]_{[d]} ) \le m(A)$.
Similarly take  $B' := [[B]]_{[d]}$ and  $C' := [[C]]_{[d]}$.
Note that $C$ is a cylinder, free in the coordinates of index greater than  $d$, so $C = [C]_{[d]}$  and $\lambda(C) = \lambda_d( C')$.   It is ``obvious'' (and we supply details in the next paragraph) that with the usual BKR operator  on $[0,1]^d$,      $A' \bkr B' = C'$, so Corollary 
 \ref{cor [0,1]} applies, showing that $\lambda_d (C') \le \lambda_d(A') \, \lambda_d(B')$, and chaining together inequalities completes the proof that  $\lambda(A \bkr B) \le m(A) \, m(B)$.
 
Details for $A' \bkr B' = C'$:  
We start with $C := A \bkr_d B$ as defined by \eqref{def bkr d}, and apply $\Proj_{[d]}$.
The relation $([A]_J)_K = [A]_{J \cap K}$ in $[0,1]^\BN$, used with $K=[d]$, shows that  for $J \subset [d]$,
$$
  ([A]_{[d]})_J  = [A]_J ,  \text{ and hence, in } [0,1]^d, [A']_J = \Proj_{[d]} ([A]_J).
$$
The function $\Pr = \Proj_{[d]}$, which is the set-to-set function induced by $\pr_{[d]}: [0,1]^\BN \to [0,1]^d$, distributes over unions.   
For $J,K \subset [d]$, $[A']_J = \Pr( [A]_J)$ and $[B']_K = \Pr( [B]_K)$,   also, both $[A]_J$ and $
[B]_K$ are cylinders free in all coordinates of index greater than $d$, so that
$\Pr( [A]_J) \cap \Pr( [B]_K) = \Pr( [A]_J \cap [B]_K)$.
Hence, with all unions taken over disjoint $J,K \subset [d]$, 
\begin{eqnarray*}
A' \bkr B' &=&\bigcup \  [A']_J \cap [B']_K \\
   & = & \bigcup \ \Pr( [A]_J) \cap \Pr( [B]_K) \\
   & = & \Pr \left( \bigcup \ [A]_J \cap [B]_K \right) \\
   & = & \Pr \left( A \bkr_d B \right) = \Pr(C) = \Pr([C]_{[d]}) = C'.
\end{eqnarray*}
Finally, the result for the simultaneous $r$-fold BKR operator follows by a similar argument,
starting with an extension of \eqref{def bkr d} to define a level-$d$  $r$-fold BKR operator.
\end{proof}

\begin{theorem}\label{thm R N}
Consider the BKR combination of events, that they occur for \emph{finite} disjoint sets of reasons,
as specified by \eqref{def bkr infty} and \eqref{def bkr r infty}.
For Borel subsets $A,B$ of  $\BR^\BN$, under any  complete product probability measure $\p$ on
$\BR^\BN$,
\begin{equation}\label{real conclusion N}
 \p(A \bkr B) \le \p(A) \, \p(B).
\end{equation}
For Borel subsets $A_1,\ldots,A_r$ of $\BR^\BN$, under any  complete product probability measure $\p$ on
$\BR^\BN$,
\begin{equation}\label{real conclusion many N}
 \p(\bigbkr_1^r A_i ) \le \prod_1^r \p(A_i).
\end{equation}
\end{theorem}
\begin{proof}
The result follows immediately from Theorem \ref{thm [0,1] N}, by adapting \eqref{def g} and the argument used to prove Theorem \ref{thm R},
from the context of $\BR^d$, to the context of $\BR^\BN$.
\end{proof}

\section{Relaxing the sample space}
\label{sect relax}
In this paper we consider a sample space $S^I$ for $I$ countable
and $S=[0,1]$ --- with Lebesgue measure on $S^I$, or $S=\BR$, with arbitrary  complete product probability measure on $S^I$.
However, all results can be carried over to the superficially more general case 
$\Omega := \prod_{i \in I} S_i$ for $S_i$ a Polish subspace (equivalently, $G_\delta$ subset) of $\BR$, each $S_i$ is endowed with a probability measure $\p_i$ defined on the Borel subsets, and $\Omega$ has the product measure $\p = \prod \p_i$.

Extend $\p_i$, $\p$ to measures $\hat{\p}_i$, $\hat{\p}$ on $\BR$, $\BR^I$ respectively by taking them to be 0 on the complement.  The definition of the BKR operation from \eqref{def bkr many} or \eqref{def bkr r infty} rephrases in a natural way to $\Omega$.  One finds, for Borel sets $A_j \subset \Omega$: a) $\bkr_j A_j$ is $\p$-measurable by the argument of Lemma \ref{lem cylinder}, and b), writing $\hat{\bkr}$ for the BKR operation computed with respect to $\BR^I$ and $\hat{A}_j := A_j \cup (\BR^I \setminus \Omega)$, that 
\[
\bkr_j A_j = \left( \hat{\bkr}_j \hat{A}_j \right) \cap \Omega.
\]
Therefore, $\p(\bkr_j A_j) = \hat{\p}(\hat{\bkr}_j \hat{A}_j )$ and $\prod_j \hat{\p}(\hat{A}_j) = \prod_j \p(A_j)$, and it is clear that Theorems \ref{thm R} and \ref{thm R N} for $\BR^I$ 
imply the BKR inequality for
$\prod_{i \in I} S_i$.

\section{From $\Omega$ to $\Omega$}

It is tempting to attempt to extend our results 
to get something symmetric,  where we assume that the inputs $A,B$ are in
a larger family of sets than the Borel sets, and the output $A \bkr B$, satisfying $\lambda_d(A \bkr B) \le \lambda_d(A) \, \lambda_d(B)$, is in the same family.
Since defining the BKR product requires only complement, countable union, 
 and projection, the ``larger family" 
should be the class of projective sets, the smallest 
extension of the class of Borel sets closed under 
projection, complement and countable union, see \cite{YNM,AK}.  Then the version of Lemma \ref{lem cylinder}, \emph{If $A,B$ are projective, then the cylinders $[A]_K$ and the BKR product $A \bkr B$ are also projective}, is immediately true.

Probabilists may be familiar with the construction of the family of Borel sets, starting from the family of open sets,  take complements and countable unions, to get a larger family, then iterate -- see \cite[pages 30--32]{Billingsley3}.   The construction of projective sets is similar;  start with the Borel sets, 
take projections, countable unions, and complements, to get a larger family, then iterate.
But there is a difference:  the construction of Borel sets requires iteration out to the first uncountable ordinal, usually denoted $\Omega$, while the construction of projective sets is finished at the first infinite ordinal $\omega$.

In view of Corollary \ref{cor [0,1]}, to get BKR inequalities, we need 
only show that Lebesgue measure extends to projective sets.  
Here the situation is somewhat complex.  It is \emph{consistent}  
with ZFC to assume that such extension is false, in fact that 
there are nonmeasurable projective sets only one level in the 
projective hierarchy above analytic sets \cite{Godel}. On the other 
hand, the existence of an inaccessible cardinal would imply 
that all projective sets are measurable \cite{solovay}. Though such 
existence cannot be proved to be consistent with ZFC, it is widely 
assumed that this (consistency) is true --- and often such existence 
is accepted as a useful extra axiom.

\section{Open problems}

\begin{problem}  
For the BKR operator $\bkr_{=\infty}$ defined by \eqref{def bkr = infty}, prove 
or give a counterexample: 
\emph{For Borel subsets $A,B$ in $[0,1]^\BN$,  there exists a Borel set $C$, with  
$A \bkr B \subset C$ and $m(C) \le m(A) \, m(B)$}.
\end{problem}

It is not hard to determine, for the special case $d=2$, when the BKR inequality holds with equality:
for Borel sets $A,B \subset [0,1]^2$, \ $\lambda_2(A \bkr B) = m_2(A) m_2(B)$ if and only if
if and only if 0) $m_1(A) m_2(B)=0$, or
1) A or B is all of $[0,1]^2$, or
2) A and B are each unions of a ``cylinder" and a measure zero set,
with the two cylinders being ``orthogonal", i.e., in different 
directions.

\begin{problem}
Give a simple necesary and sufficient condition for  $A,B \subset [0,1]^d$,  to satisfy  $\lambda_d(A \bkr B) = m_d(A) m_d(B)$.
\end{problem}

As background for problems \ref{catalan problem} and \ref{strict catalan problem}: in \cite[Prop. 5.5]{luck}, for arbitrary $S$ and $A_1,\ldots,A_r \subset S^d$, we showed that
$\bigbkr_1^r A_i \subset ( \cdots ((A_1 \bkr A_2) \bkr A_3) \cdots \bkr A_{r-1}) \bkr A_r)$. For brevity we omit the symbol for binary BKR operator, and write simply $\bigbkr_1^r A_i \subset ( \cdots ((A_1  A_2)  A_3) \cdots  A_{r-1})  A_r)$
For a binary operator, 
the number of ways to associate a product with $r$ factors
is given by the Catalan number $C_{r-1}$, and the same argument shows that the simultaneous
$r$-fold BKR product, $\bigbkr_1^r A_i$, is a subset of each of the binary-associated products.

\begin{problem}\label{catalan problem}  Prove or disprove:  for $r=3,4,\ldots$, there exist $S$ and $d$, and 
$A_1,\ldots,A_r \subset S^d$,  such that the $C_{r-1}$ binary-associated products
for $A_1A_2\cdots A_r$ are all distinct.
\end{problem}

\begin{problem}\label{strict catalan problem}  For $r=3,4,\ldots$, for  any $S$ and $d$, and for any 
$A_1,\ldots,A_r \subset S^d$,  we already know that  $\bigbkr_1^r A_i$ is a subset of 
the \emph{intersection} of  the $C_{r-1}$ binary-associated products
for $A_1A_2\cdots A_r$.   Prove or disprove:  for $r=3,4,\ldots$, there exists an example where the containment of $\bigbkr_1^r A_i$ is \emph{strict}.
\end{problem}

Now consider cases where all $r$ factors are the same set $A$.  Commutativity of the  binary BKR product implies that $(A \bkr A) \bkr A = A \bkr (A \bkr A)$,  but does not resolve the situation for $r=4$ factors. Example \ref{4 example} example does resolve the situation for $r=4$.

\begin{example}[$\ ((AA)A)A \ne (AA)(AA)$ can occur]\label{4 example}
In $\{0,1\}^6$, let $A$ be the union of the following 2-cylinders, each of which is a set of size 16:

\[
\begin{tabular}{cccc} 
11****,& **11**,& 1**0**,& *11***,\\
**00**,&  ****00,& **1**0,& ***00*.
\end{tabular}
\]

Note that the first two 2-cylinders combine to show that 1111** $\subset AA$,
the next two show that 1110** $\subset AA$.  Hence the first four 2-cylinders show that
111*** $\subset AA$.  Similarly, the last four 2-cylinders show that ***000  $\subset AA$.
Combining, we see that 111000 $\in (AA)(AA)$.  Computer-exhaustive checking shows that 
$ ((AA)A)A =\emptyset$, hence $((AA)A)A \ne (AA)(AA)$.
\end{example}

In honor of Wedderburn \cite{wedderburn}, \cite[Sequence A001190]{OEIS}, write $W_n$ for the number of ways to binary-associate a product of the form $A^n$, up to equivalence modulo the commutative  property of the binary relation;  for example, $W_2,W_3,\ldots,W_7 = 1,1,1,2,3,6,11$.

\begin{problem} 

\begin{enumerate}

\item  For $r=5,6,\ldots,$ does there exist an example with a single set $A$, such that all $W_r$ equivalence classes of association yield different results?

\item  As above, with the additional restriction that $A \subset \{0,1\}^d$ for some $d$ depending on $r$.

\item  If, for a given $r$, there is an example with $A \subset \{0,1\}^d$  such that all $W_r$ equivalence classes of association yield different results,  write $D_r$ for the smallest such $d$, following the notation  Ramsey numbers.
Example \ref{4 example} shows that $D_4 \le 6$.   Can you prove that $D_4 > 5$?
Can you determine $D_5$?   Or give nontrivial upper or lower bounds for $D_r$ for general $r$?
 
\end{enumerate}
\end{problem}

\bibliographystyle{amsalpha}
\bibliography{lottery}

\end{document}